\documentclass[11 pt]{amsart}
\usepackage{amsmath}
\usepackage{amsfonts}
\usepackage{enumerate}
\usepackage{verbatim}
\usepackage[pdftex]{graphicx}
\DeclareGraphicsExtensions{.png,}
\newtheorem{lemma}{Lemma}
\newtheorem{theorem}{Theorem}
\newtheorem{example}{Example}
\newtheorem{prop}{Proposition}
\newtheorem{corollary}{Corollary}

\newtheorem{defn}{Definition}
\numberwithin{equation}{section}
\numberwithin{defn}{section}
\numberwithin{example}{section}
\numberwithin{prop}{section}
\numberwithin{lemma}{section}

\newcommand{\R}{\mathbb{R}}

\newcommand{\N}{\mathbb{N}}

\newcommand{\V}{\mathcal{V}}

\newcommand{\Lo}{\mathcal{L}}
\newcommand{\T}{\mathcal{T}}
\newcommand{\Ss}{\mathcal{S}}
\newcommand{\Sh}{\mathcal{S}}
\newcommand{\Exp}{\mathbb{E}}
\newcommand{\A}{\mathcal{A}}
\newcommand{\C}{\mathcal{C}}
\newcommand{\W}{\mathcal{W}}
\newcommand{\Hau}{\mathcal{H}}
\newcommand{\RR}{\mathcal{R}}
\newcommand{\B}{\mathcal{B}}
\newcommand{\F}{\mathcal{F}}

\newcommand{\var}{\text{var}}

\newcommand{\D}{\mathcal{D}}
\newcommand{\Error}{\eta(q,r)}

\newcommand{\diam}{\mbox{diam}}

\renewcommand{\dim}{\mbox{dim}_{\mathcal{H}}}

\title[Shrinking Targets for Countable Markov Maps]{Shrinking Targets for Countable Markov Maps}
\author{Henry WJ Reeve}
\address{Henry WJ Reeve\\Department of Mathematics\\ The University of Bristol\\
University Walk\\Clifton\\ Bristol\\BS8 1TW\\UK}
\email{henrywjreeve@googlemail.com}
\begin{document}

\begin{abstract}
Let $T$ be an expanding Markov map with a countable number of inverse branches and a repeller $\Lambda$ contained within the unit interval. Given $\alpha \in \R_+$ we consider the set of points $x \in \Lambda$ for which $T^n(x)$ hits a shrinking ball of radius $e^{-n\alpha}$ around $y$ for infinitely many iterates $n$. Let $s(\alpha)$ denote the infimal value of $s$ for which the pressure of the potential $-s\log|T'|$ is below $s \alpha$. Building on previous work of Hill, Velani and Urba\'{n}ski we show that for all points $y$ contained within the limit set of the associated iterated function system the Hausdorff dimension of the shrinking target set is given by $s(\alpha)$. Moreover, when $\overline{\Lambda}=[0,1]$ the same holds true for all $y \in [0,1]$. However, given $\beta \in (0,1)$ we provide an example of an expanding Markov map $T$ with a repeller $\Lambda$ of Hausdorff dimension $\beta$ with a point $y\in \overline{\Lambda}$ such that for all $\alpha \in \R_+$ the dimension of the shrinking target set is zero.
\end{abstract}

\maketitle

\section{Introduction}

Suppose we have a dynamical system $(X,T,\mu)$ consisting of a space $X$ together with a map $T:X \rightarrow X$ and a $T$-invariant ergodic probability measure $\mu$. Let $A$ be a subset of positive $\mu$ measure. Poincar\'{e}'s recurrence theorem implies that $\mu$ almost every $x\in X$ will visit $A$ an infinite number of times, ie. $\bigcap_{m\in\N} \bigcup_{n\geq m}T^{-n}A$ has full $\mu$ measure. This raises the question of what happens when we allow $A$ to shrink with respect to time. How does the size of $\bigcap_{m\in\N} \bigcup_{n\geq m}T^{-n}A(n)$ depend upon the sequence $\left\lbrace A(n)\right\rbrace_{n \in \N}$? 

We shall consider this question in the setting of hyperbolic maps. Given a Gibbs measure $\mu$, Chernov and Kleinbock have given general conditions according to which $\bigcap_{m\in\N} \bigcup_{n\geq m}T^{-n}A(n)$ will have full $\mu$ measure \cite{CK}. However, when $\sum_{n=0}^{\infty} \mu(A(n))$ is finite it is clear that $\bigcap_{m\in\N} \bigcup_{n\geq m}T^{-n}A(n)$ must be of zero $\mu$ measure. In particular, if $\left\lbrace A(n)\right\rbrace_{n \in \N}$ is a sequence of balls which shrink exponentially fast around a point, then $\bigcap_{m\in\N} \bigcup_{n\geq m}T^{-n}A(n)$ must be of zero Lebesgue measure. Thus, in order to understand its geometric complexity we must determine its Hausdorff dimension (see \cite{F1} for an introduction to dimension theory).
 
In \cite{HV1, HV2} Hill and Velani consider the dimension of the shrinking target set 
\begin{equation*}
\D_y(\alpha):=\bigcap_{m\in\N} \bigcup_{n\geq m}\left\lbrace x\in X: |T^n(x)-y|< e^{-n \alpha}\right\rbrace. 
\end{equation*}
Let $s(\alpha)$ denote the infimal value of $s$ for which the pressure of the potential $-s\log|T'|$ is below $s \alpha$. In \cite{HV2} it is shown that for an expanding rational maps of the Riemann sphere the dimension of $\D_y(\alpha)$ is given by $s(\alpha)$ for all points $y$ contained within the Julia set. Now suppose we have a piecewise continuous map of the unit interval $T$ with repeller $\Lambda$. When $T$ has just finitely many inverse branches, Hill and Velani's formula for the dimension of $\D_y(\alpha)$ extends unproblematically. That is, for all $y\in \overline{\Lambda}$, $\dim \D_y(\alpha)=s(\alpha)$. However when $T$ has an infinite number of inverse branches things become more difficult, owing to the unboundedness $|T'|$. In \cite{U} Urba\'{n}ski showed that for those $y\in \Lambda$ satisfying $\sup\{|(T')(T^n(y))|\}_{n\geq 0}<\infty$, the dimension of $\D_y(\alpha)$ is equal to $s(\alpha)$. We prove that, even for systems with an infinite number of inverse branches, this formula extends to all points $y \in \Lambda$. Moreover, when $\overline{\Lambda}=[0,1]$ we have $\dim \D_y(\alpha)=s(\alpha)$ for all $y \in [0,1]$. However, we provide a family of examples showing that when $\dim \Lambda \in (0,1)$, whilst $s(\alpha)$ is always positive, the dimension of $\D_y(\alpha)$ can be zero for certain members of $y\in \overline{\Lambda}\backslash \Lambda$.

\section{Statement of results}

Before stating our main results we shall introduce some notation and provide some further background.

\begin{defn}[Expanding Markov Map]\label{EMR def} Let $\V=\{V_i\}_{i\in\A}$ be a countable family of disjoint subintervals of the unit interval with non-empty interior. Given $\omega =(\omega_0,\cdots,\omega_{n-1}) \in \A^n$ for some $n\in \N$ we let $V_{\omega}:=\cap_{\nu=0}^{n-1}T^{-\nu}V_{\omega_{\nu}}$. We shall say that $T: \cup_{i\in\A}V_i \rightarrow [0,1]$ is an expanding Markov map if $T$ satisfies the following conditions.
\begin{itemize}
\vspace{2mm}
\item[(1)] For each $i \in \A$, $T|_{V_i}$ is a $C^1$ map which maps the interior of $V_i$ onto open unit interval $(0,1)$,
\vspace{2mm}
\item[(2)] There exists $\xi>1$ and $N \in \N$ such that for all $n\geq N$ and all $x \in \cup_{\omega \in \A^n}V_{\omega}$ we have $|(T^n)'(x)|>\xi^n$,
\vspace{2mm}
\item[(3)] There exists some sequence $\{\rho_n\}_{n\in \N}\subset \R$ with
$lim_{n \rightarrow \infty }\rho_n = 0$ such that for all $n \in \N$, $\omega \in\A^n$, and all $x, y \in V_{\omega}$,
\[ e^{-n\rho_n} \leq \frac{|(T^n)'(x)|}{|(T^n)'(y)|}\leq e^{n\rho_n}.\]
\end{itemize}
\vspace{.5mm}
We shall say that $T$ is a finite branch expanding Markov map if $\A$ is a finite set. 
\end{defn}
The repeller $\Lambda$ of an expanding Markov map is the set of points for which every iterate of $T$ is well-defined, $\Lambda:=\bigcap_{n \in \N}T^{-n}([0,1])$.
We assume throughout that $\#\A>1$. Otherwise $\Lambda$ would either empty or contained within a single point.

Given a point $y\in \overline{\Lambda}$ in the closure of the repeller and some $\alpha\in \R_+$ we shall be interested in the set of points $x \in \Lambda$ for which $T^n(x)$ hits a shrinking ball of radius $e^{-n\alpha}$ around $y$ for infinitely many iterates $n$,
\begin{eqnarray}
\D_y(\alpha):=\bigcap_{m\in\N} \bigcup_{n\geq m}\left\lbrace x\in \Lambda: |T^n(x)-y|<e^{-n\alpha} \right\rbrace.
\end{eqnarray}
More generally, given a function $\varphi: \Lambda \rightarrow \R_+$ we let $S_n(\varphi):= \sum_{i=0}^{n-1}\varphi \circ T^l$ and define
\begin{eqnarray}
\D_y(\varphi):=\bigcap_{m\in\N} \bigcup_{n\geq m}\left\lbrace x\in \Lambda: |T^n(x)-y|<e^{-S_n(\varphi)(x)} \right\rbrace.
\end{eqnarray}

Sets of the form $\D_y(\varphi)$ arise naturally in Diophantine approximation.

\begin{example}\label{Gauss example}
Given $\alpha \in R_+$ we let
\begin{equation*}
J(\alpha):=\left\lbrace x\in [0,1]:\bigg|x-\frac{p}{q}\bigg|<\frac{1}{q^{\alpha}}\text{ for infinitely many }p,q \in \N\right\rbrace.
\end{equation*}
Let $T:[0,1]\rightarrow [0,1]$ be the Gauss map $x \mapsto \frac{1}{x}-\lfloor \frac{1}{x}\rfloor$ which is an expanding Markov map on the repeller $\Lambda=[0,1]\backslash \mathbb{Q}$. We define $\psi:\Lambda \rightarrow \R$ by $\psi(x)= \log |T'(x)|$ and for each $\alpha >2$ we let $\psi_{\alpha}:=\left(\frac{\alpha}{2}-1\right)\psi$. Then for all $2<\alpha<\beta<\gamma$ we have,
\begin{equation}\label{diophantine containment}
\D_0\left(\psi_{\alpha}\right)\subset J(\beta) \subset \D_0\left(\psi_{\gamma}\right).
\end{equation}
In \cite{J, B} Jar\'{n}ik and Besicovitch showed that for $\alpha>2$, $\dim(J(\alpha))=\frac{2}{\alpha}$. By 
(\ref{diophantine containment}) this is equialent to the fact that for all $\alpha>2$ 
\begin{equation*}
\dim D_0\left(\psi_{\alpha}\right)=\frac{2}{\alpha}.
\end{equation*}
\end{example}

As we shall see, in sufficiently well behaved settings, the Hausdorff dimension of $\D_y(\varphi)$ may be expressed in terms of the thermodynamic pressure. 

\begin{defn}[Tempered Distortion Property] 
Given a real-valued potential $\varphi:\Lambda \rightarrow \R$ we define the $n$-th level variation of $\varphi$ by,
\begin{eqnarray*}
\var_n(\varphi):=\sup\left\lbrace |\varphi(x)-\varphi(y)|: x,y \in V_{\omega}, \omega \in \A^n \right\rbrace.
 \end{eqnarray*}
We shall say that a potential $\varphi$ satisfies the tempered distortion condition if $\var_1(\varphi)<\infty$ and $\lim_{n \rightarrow \infty}n^{-1}\var_n(S_n(\varphi))=0.$ 
\end{defn}

Note that by condition (3) in definition \ref{EMR def} the potential $\psi(x):=\log |T'(x)|$ satisfies the tempered distortion condition.

Given a potential $\varphi: \Lambda \rightarrow \R$ and a word $\omega \in \A^n$ for some $n \in \N$ we define $
\varphi(\omega):=\sup \left\lbrace \varphi(x): x \in V_{\omega} \right\rbrace.$

\begin{defn}\label{pressure def}Given a potential $\varphi:\Lambda \rightarrow \R$, satisfying the tempered distortion condition, we define the pressure by
\begin{eqnarray*}
P(\varphi):=\lim_{n\rightarrow \infty}\frac{1}{n}\log \sum_{\omega \in \A^n}\exp(S_n(\varphi)(\omega)).
\end{eqnarray*}
\end{defn}
This definition of pressure is essentially the same as that given by Mauldin and Urba\'{n}ski in \cite{MU, GDMS}. We note that the limit always exists, but may be infinite.
Recall that we defined $\psi(x)$ to be the log-derivative, $\psi(x):=\log |T'(x)|$. Given $\alpha>0$ we define $s(\alpha)$ by,
\begin{equation}
s(\alpha):=\inf\left\lbrace s: P(-s \psi)\leq s\alpha \right\rbrace.
\end{equation}
More generally, given a non-negative positive potential $\varphi:\overline{\Lambda}\rightarrow \R_{\geq 0}$, satisfying the tempered distortion condition, we define,
\begin{equation}
s(\varphi):=\inf\left\lbrace s: P(-s (\psi+\varphi))\leq 0 \right\rbrace.
\end{equation}

The project of trying to determine the Hausdorff dimension of $\D_{y}(\varphi)$ began with a series of articles due to Hill and Velani \cite{HV1, HV2, HV3}. Whilst Hill and Velani gave the dimension of $\D_{y}(\varphi)$ for an expanding rational map of the Riemann sphere, the result extends unproblematically to any expanding Markov map with finitely many inverse branches.

\begin{theorem}[Hill, Velani]\label{HV Theorem} Let $T$ be a finite branch expanding Markov map with repeller $\Lambda$ and let $\varphi:\Lambda \rightarrow \R$ a non-negative potential which satisfies the tempered distortion condition. Then, for all $y \in \overline{\Lambda}$ we have $\dim \D_{y}(\varphi)=s(\varphi)$.
\end{theorem}

Given the neat connection between Diophantine approximation and shrinking target sets for the Gauss map it is natural to try to generalise Theorem \ref{HV Theorem} to the setting of expanding Markov maps with an infinite number of inverse branches. However, for such maps things can become much more delicate.

Note that we always have $\Lambda_{\circ}\subseteq \Lambda \subseteq \overline{\Lambda}$. Indeed, when $T$ is a finite branch Markov map $\Lambda_{\circ}= \Lambda = \overline{\Lambda}$, up to a countable set. However, for Markov maps with infinitely many inverse branches both of these containments may be strict.

In \cite{U} Urba\'{n}ski proves the following extention of Theorem \ref{HV Theorem} to points $y \in \Lambda_{\circ}$ for an infinite branch expanding Markov map.

\begin{theorem}[Urba\'{n}ski]\label{U theorem} Let $T$ be an expanding Markov map with repeller $\Lambda$ and let $\varphi:\Lambda \rightarrow \R$ a non-negative potential which satisfies the tempered distortion condition. Then, for every $y \in \Lambda_{\circ}$ we have $\dim \D_{y}(\varphi)=s(\varphi)$.
\end{theorem}
In terms of dimension $\Lambda_{\circ}$ is a large set, with $\dim \Lambda_{\circ}= \dim \Lambda$ \cite{MU}. However, it follows from Bowen's equation combined with the strict monotonicity of the pressure function for finite iterated function systems (see \cite[Chapter 5]{F2}) that for any $T$ ergodic measure with $\dim \mu=\dim \Lambda$, $\mu(\Lambda_{\circ})=0$. For example, when $T$ is the Gauss map and $\mathcal{G}$ the Gauss measure, which is ergodic and equivalent to Lebesgue measure $\mathcal{L}$, then $\Lambda_{\circ}$ is the set of badly approximable numbers with $\dim \Lambda_{\circ}=1$ and $\mathcal{L}(\Lambda_{\circ})=\mathcal{G}(\Lambda_{\circ})=0$.

Our main theorem extends the above result to all $y\in \Lambda$.
\begin{theorem}\label{ifs} Let $T$ be an expanding Markov map with repeller $\Lambda$ and let $\varphi:\Lambda \rightarrow \R$ be a non-negative potential which satisfies the tempered distortion condition.  Then, for every $y \in \Lambda$ we have $\dim \D_{y}(\varphi)=s(\varphi)$.
\end{theorem}

Note that in Example \ref{Gauss example} $0 \notin \Lambda=\R\backslash \mathbb{Q}$, so it is clear that for certain maps $\dim D_{y}(\varphi)=s(\varphi)$ holds for $y \in \overline{\Lambda} \backslash \Lambda$. The following theorem shows that this holds whenever $\Lambda$ is dense in the unit interval.

\begin{theorem}\label{dense unit interval} Let $T$ be an expanding Markov map with a repeller $\Lambda$ satisfying $\overline{\Lambda}=[0,1]$ and let $\varphi:\Lambda \rightarrow \R$ a non-negative potential which satisfies the tempered distortion condition. Then, for every $y \in [0,1]$ we have $\dim D_{y}(\varphi)=s(\varphi)$.
\end{theorem}

Returning to Example \ref{Gauss example} we let $T$ denote the Gauss map and $\psi_{\alpha}:=\left(\frac{\alpha}{2}-1\right)\psi$ and let $\alpha>2$. By the Jar\'{n}ik Besicovitch theorem \cite{J, B} we have $\dim D_0\left(\psi_{\alpha}\right)=\frac{2}{\alpha}$. It follows from Theorem \ref{U theorem} \cite{U} that $\dim \D_y\left(\psi_{\alpha}\right)=\frac{2}{\alpha}$ also holds for all badly approximable numbers $y$. By Theorem \ref{dense unit interval} we see that $\dim \D_y\left(\psi_{\alpha}\right)=\frac{2}{\alpha}$ for all $y \in [0,1]$.

We remark that Bing Li, BaoWei Wang, Jun Wu, Jian Xu have independently obtained a proof of Theorem \ref{dense unit interval} in the special case in which $T$ is the Gauss map, as well some interesting results concerning targets which shrink at a super-exponential rate \cite{BBJJ}. However, the methods used in \cite{BBJJ} rely upon certain properties of continued fractions which do not hold in full generality.

Now suppose that $\overline{\Lambda} \neq [0,1]$ and $y \in \overline{\Lambda}\backslash \Lambda$. It might seem reasonable to conjecture that again $\dim \D_{y}(\varphi)=s(\varphi)$. However this is not always the case and, as the following theorem demonstrates, this conjecture fails in rather a dramatic way.

Given $\Phi:\N \rightarrow \R_+$ we define,
\begin{eqnarray*}
\Ss_y(\Phi):=\bigcap_{m\in\N} \bigcup_{n\geq m}\left\lbrace x\in X : d(T^n(x),y)< \Phi(n) \right\rbrace.
\end{eqnarray*}

\begin{theorem}\label{CE} Let $\Phi: \N \rightarrow \R_{>0}$ be any strictly decreasing function satisfying $\lim_{n \rightarrow \infty} \Phi(n)=0$. Then, for each $\beta \in (0,1)$ there exists an expanding Markov map $T$ with a repeller $\Lambda$ with $\dim \Lambda = \beta$ together with a point $y \in \overline {\Lambda}$ satisfying $\dim \Sh_{y}(\Phi)=0$.
\end{theorem}
Thus, even for $\Phi$ which approaches zero at a subexponential rate we can have $\dim \Sh_{y}(\Phi)=0$. We remark that $s(\alpha)$ is always strictly positive.

We begin In Section \ref{UBS} we prove the upper bound in Theorems \ref{ifs} and \ref{dense unit interval} simultaneously with an elementary covering argument. In Section \ref{LBS} we introduce and prove a technical proposition which implies the lower bounds in both Theorems \ref{ifs} and \ref{dense unit interval}. In Section \ref{CES} we prove Theorem \ref{CE}. We conclude in Section \ref{remarks} with some remarks.

\section{Infinite iterated function systems}

In order to make the proof more transparent we shall employ the language of iterated function systems.

Let $T:\cup_{i \in \A}V_i\rightarrow [0,1]$ be a countable Markov map. We associate an iterated function system $\left\lbrace \phi_i \right\rbrace_{i \in \A}$ corresponding to $T$ in the following way. For each $i \in \A$ we let $\phi_i:[0,1] \rightarrow \overline{V}_i$ denote the unique $C^1$ map satisfying $\phi_i \circ T(x) = x$ for all $x \in V_i$. 

Let $\Sigma$ denote symbolic space $\A^{\N}$ endowed with the product topology and let $\sigma: \Sigma \rightarrow \Sigma$ denote the left shift operator. Given an infinite string $\omega=(\omega_{\nu})_{\nu\in\N} \in \Sigma$ and $m,n\in \N$ we let $m|\omega|n$ denote the word $(\omega_{\nu})_{\nu=m+1}^n\in \A^{n-m}$. Given $\tau=(\tau_1,\cdots,\tau_n)\in\A^n$ for some $n\in\N$ we let $\phi_{\tau}:=\phi_{\tau_1}\circ \cdots \circ \phi_{\tau_n}$. Sets of the form $\phi_{\tau}([0,1])$ are referred to as cylinder sets.

Take $\omega \in \A^{\N}$. Note that by definition \ref{EMR def} (2) we have $\diam (\phi_{\omega_n}([0,1]))\leq \xi^{-n}$ for all $n\geq N$. Thus, we may define,
\begin{equation*}
\pi(\omega):= \bigcap_{n\in \N} \phi_{\omega|n}([0,1]).
\end{equation*}
This defines a continuous map $\pi:\Sigma \rightarrow [0,1]$. 

Since the intervals $\left\lbrace V_i\right\rbrace_{i \in \A}$ have disjoint interiors the iterated function system $\{\phi_i\}_{i \in \A}$ satisfies the open set condition (see \cite[Section 9.2]{F1}) and $\pi(\Sigma)\backslash \Lambda$ is countable. By definition \ref{EMR def} (1) we have $T\circ \pi(\omega)=\pi\circ\sigma (\omega)$ for all $\omega\in \pi^{-1}\left(\Lambda\right)$. Thus, $T: \Lambda \rightarrow \Lambda$ and $\sigma: \Sigma \rightarrow \Sigma$ are conjugate up to a countable set.

In Definition \ref{pressure def} we have used a slightly modified version of the definition given in \cite[(2.1)]{GDMS}. Nevertheless, the following theorems may be proved in essentially the same way as the proofs given in \cite{GDMS}.

\begin{theorem}[Mauldin, Urba\'{n}ski]\label{Bowen's equation} Given a countable Markov map $T$ with repeller $\Lambda$ we have $\dim \Lambda=\inf\left\lbrace s:P(-s\psi)\leq 0\right\rbrace$.
\end{theorem}
When $T$ has finitely many branches there is a unique $s(\Lambda)$ such that $P(-s(\Lambda)\psi)=0$ and $\dim \Lambda= s(\Lambda)$. However, Mauldin and Urba\'{n}ski have shown that when $T$ has countably many inverse branches we can have $P(-t \psi) < 0$ for all $t \geq \inf\left\lbrace s:P(-s\psi)\leq 0\right\rbrace$ and consequently there is no such $s(\Lambda)$ (see \cite[Example 5.3]{MU}). Similar examples show that in general there need not be any $s$ satisfying $P(-s(\psi+\varphi))=0$ and consequently we must take $s(\varphi):=\inf\left\lbrace s: P(-s (\psi+\varphi))\leq 0 \right\rbrace$ in Theorems \ref{ifs} and \ref{dense unit interval}.

The pressure $P$ has the following finite approximation property.
\begin{theorem}[Mauldin, Urba\'{n}ski]\label{Finite Approx prop} Let $T$ be a countable Markov map and $\varphi: \Lambda \rightarrow \R$ a potential satisfying the tempered distortion condition. Then $P(\varphi)=\sup\left\lbrace P_{\F}(\varphi): \F \subseteq \A \text{ is a finite set}\right\rbrace.$
\end{theorem}
\begin{corollary}\label{useful press corollary} Let $\varphi:\Lambda \rightarrow \R$ be a non-negative potential satisfying the tempered distortion condition. Then $P(-s(\varphi)(\psi+\varphi))\leq 0$.
\end{corollary}
\begin{proof}
Suppose $P(-s(\varphi)(\psi+\varphi))>0$. Then, by Theorem \ref{Finite Approx prop}. $P_{\F}(-s(\varphi)(\psi+\varphi))>0$ for some finite set $\F \subset \A$. However $\psi+\varphi$ is bounded on $\F^{\N}$ as $\var_1(\psi),\var_1(\varphi)< \infty$, and hence $s \mapsto P_{\F}(-s(\varphi)(\psi+\varphi))$ is continuous. Thus, there exists $t>s(\varphi)$ for which 
\begin{equation*}
P(-t(\psi+\varphi))>0 \geq P_{\F}(-t(\psi+\varphi))>0.
\end{equation*}
Since $\psi+\varphi\geq 0$, $s\mapsto P(-s(\psi+\varphi))$ is non-increasing and hence, $t \leq \inf\left\lbrace s: P(-s (\psi+\varphi))\leq 0 \right\rbrace$. Since $s(\varphi)<t$ this is a contradiction.
\end{proof}

\begin{corollary}\label{positive corollary} Let $T$ be a countable Markov map. Then for all potentials $\varphi: \Lambda \rightarrow \R$,  satisfying the tempered distortion condition, $s(\varphi)>0$.
\end{corollary}
\begin{proof} Since $\psi+\varphi\geq 0$ and $\#\A\geq 2$ it follows from Defintion \ref{pressure def} that $P(-s(\psi+\varphi))\geq \log 2>0$ for all $s\leq 0$. If, however, $s(\varphi) \leq 0$ then by Corollary \ref{useful press corollary} there exists some $s\leq 0$ with $P(-s(\psi+\varphi))\leq 0$, which is a contradiction.
\end{proof}

\section{Proof of the upper bound in Theorems \ref{ifs} and \ref{dense unit interval}}\label{UBS}

In this section we use a standard covering argument to prove a uniform upper bound on the dimension of $D_{y}(\varphi)$, which entails the upper bounds in Theorems \ref{ifs} and \ref{dense unit interval}.

Throughout the proof we shall let $\rho_n$ denote
\begin{equation*}
\rho_n:=\max \left\lbrace \var_n(A_n(\psi)),\var_n(A_n(\varphi))\right\rbrace. 
\end{equation*}
Since both $\psi$ and $\varphi$ satisfy the tempered distortion condition, $\lim_{n \rightarrow \infty} \rho_n=0$.

\begin{prop} For every $y \in [0,1]$ we have $\dim D_{y}(\varphi)\leq s(\varphi)$.
\end{prop}
\begin{proof}
For each $n\in \N$ and $\omega \in \A^n$ we define,
\begin{equation}
V_{\omega}^{\varphi,n}:=\left\lbrace x\in V_{\omega}: |T^n(x)-y|< e^{-\inf_{z\in V_{\omega}}S_n(\varphi)(z)}\right\rbrace.
\end{equation}
Clearly every $x\in \D_y(\varphi)$ is in $V_{\omega}^{\varphi,n}$ for infinitely many $n\in \N$ and $\omega \in \A^n$. Moreover, by the mean value theorem we have,
\begin{eqnarray}\label{diam}
\diam(V_{\omega}^{\varphi,n} )&\leq& e^{-\inf_{z\in V_{\omega}}S_n(\phi)(z)-\inf_{z\in V_{\omega}}S_n(\varphi)(z)}\\
\nonumber &\leq& e^{-\inf_{z\in V_{\omega}}S_n(\phi)(z)-\inf_{z\in V_{\omega}}S_n(\varphi)(z)}\\
\nonumber &\leq& e^{\sup_{z\in V_{\omega}}S_n(-(\phi+\varphi))(z)+2n\rho_n}\\
\nonumber &\leq& e^{S_n(-(\phi+\varphi))(\omega)+2n\rho_n}.
\end{eqnarray}
Choose $s>s(\varphi)$, so there exists some $t<s$ with $P(-t(\phi+\varphi))\leq 0$. By condition (2) in definition \ref{EMR def} together with $\varphi\geq 0$ we have $S_n(\phi+\varphi)\geq n\log\xi$ for all sufficiently large $n$ and hence $P(-s(\phi+\varphi))< 0$. Take $\epsilon>0$ with $\epsilon<-P(-s(\phi+\varphi))$. Since $\lim_{n\rightarrow \infty}\rho_n=0$ there exists some $n_0\in\N$ such that for all $n\geq n_0$ we have,
\begin{equation}\label{presscover}
\sum_{\omega \in\A^n} \left\lbrace \exp(S_n(-s(\phi+\varphi))(\omega))\right\rbrace<e^{-n\epsilon-2n s\rho_n}.
\end{equation}
Now choose some $\delta>0$. Since $\rho_n\rightarrow 0$ and $S_n(\phi+\varphi)\geq n\log\xi$ for all sufficiently large $n$, it follows from (\ref{diam}) that we may choose $n_1\geq n_0$ so that for all $n\geq n_1$ $\diam (V_{\omega}^{\varphi,n})<\delta$. Moreover, $\bigcup_{n\geq n_1}\left\lbrace V_{\omega}^{\varphi,n}\right\rbrace_{\omega\in \A^n}$ forms a countable cover of $\D_y(\varphi)$. Applying (\ref{diam}) together with (\ref{presscover}) we see that for all $n_1\geq n_0$,
\begin{eqnarray*}
\sum_{n \geq n_1}\sum_{\omega\in \A^n}\diam (V_{\omega}^{\varphi,n})^s 
&\leq& \sum_{n \geq n_1}\sum_{\omega\in \A^n} e^{\sup_{z\in V_{\omega}}S_n(-s(\varphi+\phi))(z)+2n s\rho_n}\\
&\leq &\sum_{n\geq n_1} e^{-n\epsilon}\leq \sum_{n\geq n_0} e^{-n\epsilon} <\infty.
\end{eqnarray*}
Thus, $\Hau^s_{\delta}(\D_y(\varphi))\leq \sum_{n\geq n_0} e^{-n\epsilon}$ for all $\delta>0$ and hence $\Hau^s(\D_y(\varphi))\leq \sum_{n\geq n_0} e^{-n\epsilon}<\infty$. Thus, $\dim (\D_y(\varphi))\leq s$ and since this holds for all $s>s(\varphi)$ we have $\dim(\D_y(\varphi))\leq s(\varphi)$.
\end{proof}

\section{Proof of the lower bound in Theorems \ref{ifs} and \ref{dense unit interval}}\label{LBS}

In order to prove the lower bound to Theorems \ref{ifs} and \ref{dense unit interval} we shall introduce the positive upper cylinder density condition. The condition essentially says that there is a sequence of arbitrarily small balls, surrounding a point $y\in [0,1]$, such that each ball contains a collection of disjoint cylinder sets who's total length is comparable to the diameter of the ball. As we shall see, given any countable Markov map $T$ with repeller $\Lambda$ this condition is satisfied for all $y\in \Lambda$, and if $\overline{\Lambda}=\Lambda$, this condition is satisfied for all $y\in [0,1]$. The substance of the proof lies in showing that for any point $y\in [0,1]$, for which the positive upper cylinder density condition is satisfied, we have $\dim \D_y(\varphi)\geq s(\varphi)$.

\begin{defn}[Positive upper cylinder density]\label{PUCD} Suppose we have an expanding Markov map with a corresponding iterated function system $\left\lbrace \phi_i\right\rbrace_{i\in\A}$. Given $y \in \overline{\Lambda}$, $n\in \N$ and $r>0$ we define,
\begin{equation*}
C(y,n,r):=\left\lbrace \phi_{\tau}([0,1]):\tau \in \A^{n}, \phi_{\tau}([0,1])\subset B(y,r)\right\rbrace.
\end{equation*}
We shall say that the iterated function system $\left\lbrace \phi_i\right\rbrace_{i\in\A}$ has positive upper cylinder density at $y$ if there is a family of natural numbers $(\lambda_r)_{r \in \R_+}$ with $\lim_{r \rightarrow 0}\lambda_r=\infty$ and $\limsup_{r\rightarrow 0}\lambda_r^{-1}\log r<0$, for which
\begin{equation*}
\limsup_{r\rightarrow 0} r^{-1}\sum_{A\in C(y,\lambda_r,r)}\diam(A)>0. 
\end{equation*}
\end{defn}

\begin{prop}\label{Technical Lemma} Let $T$ be an expanding Markov map with associated iterated function system $\left\lbrace \phi_i\right\rbrace_{i\in\A}$. Suppose that $\left\lbrace \phi_i\right\rbrace_{i\in\A}$ has positive upper cylinder density at $y \in \overline{\Lambda}$. Then for each non-negative potential $\varphi: \Lambda \rightarrow \R$ which satisfies the tempered distortion condition we have $\dim \D_y(\varphi)\geq s(\varphi)$.
\end{prop}
Combining Proposition \ref{Technical Lemma} with Lemmas \ref{ifs lemma} and \ref{dense unit interval lemma} completes the proof of the lower bound in Theorems \ref{ifs} and \ref{dense unit interval}, respectively.

\begin{lemma}\label{ifs lemma}
Let $T$ be an expanding Markov map. Then the corresponding iterated function system $\{\phi_i\}_{i \in \A}$ has positive upper cylinder density at every $y \in \Lambda$.
\end{lemma}
\begin{proof}
Suppose that $y \in \Lambda$. Then there exists some $\omega \in \Sigma$ such that $y \in \phi_{\omega|n}([0,1])$ for all $n \in \N$. We shall define $(\lambda_r)_{r\in \R_+}$ by
\begin{equation*}
\lambda_r:=\min\left\lbrace n\in \N: 2\diam\left(\phi_{\omega|n}([0,1])\right)\leq r\right\rbrace.
\end{equation*}
Clearly $\lim_{r\rightarrow 0}\lambda_r=\infty$. Moreover,
\begin{equation*}
r < 2\diam\left(\phi_{\omega|\lambda_r-1}([0,1])\right)\leq 2\zeta^{-\lambda_r+1},
\end{equation*}
so $\limsup_{r\rightarrow \infty}\lambda_r^{-1}\log r\leq -\log \xi<0$.

Given any $n \in \N$ choose $r_n:=2\diam\left(\phi_{\omega|n}([0,1])\right)$. Clearly $\lambda_{r_n}=n$ and $\phi_{\omega|n}([0,1])\in C(y,n,r_n)$. Hence, 
\begin{equation*}
\limsup_{r\rightarrow 0} r^{-1}\sum_{A\in C(y,\lambda_r,r)}\diam(A)\geq \frac{1}{2}. 
\end{equation*}
\end{proof}

\begin{lemma}\label{dense unit interval lemma}
Suppose $T$ is an expanding Markov map with $\overline{\Lambda}=[0,1]$. Then the corresponding iterated function system $\{\phi_i\}_{i \in \A}$ has positive upper cylinder density at every $y \in [0,1]$.
\end{lemma}
\begin{proof}
Suppose $T$ satisfies $\overline{\Lambda}=[0,1]$. Then for any $n \in \N$ we have
\begin{equation}\label{union interval}
[0,1] \subseteq \overline{\Lambda} \subseteq \overline{\Lambda} \subseteq \overline{\bigcup_{\omega \in \A^n}\phi_{\omega}(\Lambda)}\subseteq \overline{\bigcup_{\omega \in \A^n} \phi_{\omega}([0,1])}.
\end{equation}
We define $(\lambda_r)_{r \in \R_+}$ by
\begin{eqnarray*}
\lambda_r:=\left\lceil \frac{-\log r+ \log 2}{\log \xi} \right\rceil .
\end{eqnarray*}
Clearly $\lim_{r \rightarrow 0}\lambda_r=\infty$ and $ \limsup_{r\rightarrow 0} \lambda_r^{-1}\log r =-\log \xi<0$. 

Suppose $y \in [0, \frac{1}{2}]$. Given any $r< \frac{1}{2}$ and any $\omega \in \A^{\lambda_r}$ we have 
\begin{eqnarray}\label{small intervals}
\diam \left(\phi_{\omega}([0,1])\right) \leq \xi^{-\lambda_r}< r/2. 
\end{eqnarray}
Now $C(y,n,r)$ contains all but the right most member of 
\begin{eqnarray*}
\mathcal{I}:=\left\lbrace \phi_{\omega}([0,1]): \phi_{\omega}([0,1])\cap [y,y+r)\neq \emptyset\right\rbrace, 
\end{eqnarray*} 
if such a member exists. By (\ref{union interval}) $\sum_{A \in \mathcal{I}}\diam(A) \geq r$, so by (\ref{small intervals}) we have,
\begin{equation}\label{C equation}
\sum_{A\in C(y,\lambda_r,r)}\diam(A)\geq r/2. 
\end{equation}
By symmetry \ref{C equation} also holds for $y\in [\frac{1}{2},1]$.

Letting $r \rightarrow 0$ proves the lemma.
\end{proof}

Before going into details we shall give a brief outline of the proof of Proposition \ref{Technical Lemma}. We begin by taking $s<s(\varphi)$ and extracting a certain finite set of words $\B$ such that $P_{\B}(-s(\phi+\varphi))>0$. In addition, we take a Bernoulli measure $\mu$ supported on $\B^{\N}$ with $h(\mu)=t\int(\phi+\varphi)d\mu$ for some $t>s$. We then construct a tree structure, iteratively, in the following way. Let $\Gamma_{q-1}$ be the finite collection of words in the tree at stage $q-1$ and $\gamma_{q-1}$ denote the length of those words. At stage $q$ we take $\alpha_q$ so large that $\alpha_q^{-1}\max\left\lbrace S_{\gamma_{q-1}}(\psi)(\omega), S_{\gamma_{q-1}}(\varphi)(\omega):\omega\in \Gamma_q\right\rbrace$ is negligible. We then take a ball of radius $B(y,r_q)$ so that $r_q<\exp(-\alpha_q\int \varphi d\mu)$ and $B(y,r_q)$ contains a collection of disjoint cylinder sets who's total width is comparable to $r_q$, corresponding to a finite collection of words $\RR_q$ of length $\lambda_q$. This is made possible by the upper cylinder density condition. We then choose $\beta_q$ so that $\exp(-\beta_q\int \varphi d\mu)$ is greater than, but comparable with, $r_q$. $\Gamma_q$ consists of all continuations of $\Gamma_{q-1}$ of length $\gamma_q:=\beta_q+\lambda_q$ so that  $\beta_q|\omega \in \RR_q$ and $\omega_{\nu}$ is chosen freely from $\B$ for all $\gamma_{q-1}<\nu\leq \beta_q$. Having constructed our tree we shall define $S$ to be a certain subset of its limit points for which $\omega|\beta_q$ behaves ``typically'' with respect to $\mu$ for each $q$. Given $\omega \in S$ we have $S_{\beta_q}(\varphi)(\pi(\omega))\approx \beta_q \int \varphi d\mu<-\log r_q$ so $\beta_q|\omega|\gamma_q\in \RR_q$ implies $|T^{\beta_q}(\pi(\omega))-y|<\exp(-S_{\beta_q}(\varphi)(\pi(\omega)))$. Hence $\pi(S)\subset \D_y(\varphi)$. At each stage $\beta_q$, $S$ consists of approximately $\beta_qh(\mu)$ intervals of diameter approximately $\exp(-\beta_q \int \psi d\mu)$. Moreover, for all $\omega\in S$, $\beta_q|\omega|\gamma_q\in \RR_q$. The total diameter of cylinders corresponding to words from $\RR_q$ is about $r_q \approx \exp(-\beta_q\int \varphi d\mu)$, and so at stage $\gamma_q$ $S$ consists of approximately $\beta_q h(\mu)$ intervals of diameter roughly $\exp(-\beta_q\int (\psi+\varphi) d\mu)$, giving an optimal covering exponent of $t>s$. The fact that $\beta_q \geq \alpha_q$ will be shown to imply that we cannot obtain a cover which is more efficient, and as such $\dim \pi(S)\geq t$.

\begin{proof}[Proof of Proposition \ref{Technical Lemma}]
Choose $s<s(\varphi)$ so that $P(-s(\phi+\varphi))>0$. Without loss of generality we may assume that $s>0$. Now take $\epsilon \in (0,P(-s(\phi+\varphi)))$. Since $\lim_{n \rightarrow \infty}\rho_n =0$, it follows from the definition of pressure that for all sufficiently large $n$ we have,
\begin{equation}
\sum_{\omega \in \A^n} \exp(S_n(-s(\psi+\varphi))(\omega)) > e^{\epsilon n +2 n s \rho_n}.
\end{equation}
Consequently, for all sufficiently large $n$ we have,
\begin{equation}
\sum_{\tau \in\A^n}e^{-s\left(S_n(\psi)(\tau)+S_n(\phi)(\tau)\right)} > e^{\epsilon n}.
\end{equation}
By choosing some large $k$ we obtain,
\begin{equation}
\sum_{\tau \in\A^k}e^{-s\left(S_k(\psi)(\tau)+S_k(\phi)(\tau)\right)}> 6.
\end{equation}
Thus, there exists some finite subset $\F\subseteq \A^k$ with
\begin{equation}
\sum_{\tau \in\F}e^{-s\left(S_k(\psi)(\tau)+S_k(\phi)(\tau)\right)}> 6.
\end{equation}
Note that $s>0$ and for each $\tau \in \F$, $S_k(\psi)(\tau)>0$ and $S_k(\varphi)(\tau)>0$, so $e^{-s\left(S_k(\psi)(\tau)+S_k(\phi)(\tau)\right)}\in (0,1)$ for every $\tau \in \F$.

The finite set $\F$ inherits an order $<_*$ from the order on $[0,1]$ in a natural way by $\tau_1<_* \tau_2$ if and only if $\sup \phi_{\tau_1}([0,1])\leq \inf \phi_{\tau_2}([0,1])$. Partition $\F$ into two disjoint sets $\F_1$ and $\F_2$ so that if $\tau \in \F_1$ then its succesor under $<_*$ is in $\F_2$ and if $\tau \in \F_2$ then its succesor under $<_*$ is in $\F_1$. Clearly we may choose one $m \in \{1,2\}$ so that
\begin{equation}
\sum_{\tau \in \F_{m}}e^{-s\left(S_k(\psi)(\tau)+S_k(\phi)(\tau)\right)}\geq \frac{1}{2} \sum_{\tau \in\F}e^{-s\left(S_k(\psi)(\tau)+S_k(\phi)(\tau)\right)} > 3.
\end{equation}
Since $s>0$, $S_k(\psi)(\tau)>0$ and $S_k(\varphi)(\tau)\geq 0$,  $e^{-s\left(S_k(\psi)(\tau)+S_k(\phi)(\tau)\right)}<1$ for every $\tau \in \F$. Thus we may remove both the smallest and the largest element from $\F_m$, under the order $<_*$, to obtain a set $\B\subset \F_m$ satisfying
\begin{equation}
\sum_{\tau \in \B}e^{-s\left(S_k(\psi)(\tau)+S_k(\phi)(\tau)\right)}> 1.
\end{equation}
Let $c:= \max \left\lbrace S_k(\psi)(\tau)+S_k(\varphi)(\tau):\tau \in \F \right\rbrace>0$. Given any $\omega_1, \omega_2\in \A^n$ and $\tau_1, \tau_2\in \B$ with either $\omega_1\neq \omega_2$ or $\tau_1\neq \tau_2$, or both, we have,
\begin{eqnarray}\label{D Distance Equation} |x-y| \geq \max \left\lbrace e^{-S_n(\psi)(\omega_1)-c}, e^{-S_n(\psi)(\omega_2)-c}\right\rbrace
\end{eqnarray}
for all $x\in (\phi_{\omega_1}\circ \phi_{\tau_1})([0,1])$ and $y\in (\phi_{\omega_1}\circ \phi_{\tau_1})([0,1])$. When $\omega_1 \neq \omega_2$ this follows from the fact that $\B$ contains neither the maximal nor the minimal element of $\F$ under $<_*$. When $\omega_1=\omega_2$ but $\tau_1 \neq \tau_2$ this follows from the fact that since $\tau_1, \tau_2 \in \B \subset \F_m$, $\tau_1$ cannot be the successor of $\tau_2$ and $\tau_2$ cannot be the successor of $\tau_1$.

Since $\B$ is finite and for each $\omega \in \Sigma$ $S_k(\psi)(\omega) \geq k \log \xi$ and $S_k(\psi)(\omega) \geq 0$, we may take $t\in (s,1)$ satisfying
\begin{equation}
\sum_{\tau \in\B}e^{-t\left(S_k(\psi)(\tau)+S_k(\phi)(\tau)\right)}= 1.
\end{equation}

We define a $k$-th level Bernoulli measure $\mu$ on $\B^{\N}$ by defining $p(\tau)$ for $\tau \in \A^k$ by $p(\tau):=e^{-t\left(S_k(\psi)(\tau)+S_k(\phi)(\tau)\right)}$ and setting $\mu\left([\tau_1, \cdots, \tau_n]\right)= p_{\tau_1}\cdots p_{\tau_n}$ for each $(\tau_1, \cdots, \tau_n) \in \B^{n}$. We define,
\begin{eqnarray*}
\Exp(S_k(\psi))&:=& \sum_{\tau \in \B}p(\tau)S_k(\psi)(\tau)\\
\Exp(S_k(\varphi))&:=& \sum_{\tau \in \B}p(\tau)S_k(\varphi)(\tau).
\end{eqnarray*}
Choose a decreasing sequence $\{\delta_q\}_{q \in \N} \subset \R_{>0}$ so that $\prod_{q \in \N}\left(1-\delta_q\right)>0$.
Take $q \in \N$. By Kolmogorov's strong law of large numbers combined with Egorov's theorem there exists set $S_q \subseteq \B^{\N}$ with $\mu(S_q)>1-\delta_q$ and $N(q)\in \N$ such that for all $\omega =(\omega_{\nu})_{\nu \in \N} \in S_q$ with $\omega_{\nu} \in \B$ for each $\nu \in \N$ and all $n\geq N(q)$ we have,
\begin{eqnarray}\label{KLimits1}
\frac{1}{n}\sum_{\nu=1}^{n}S_k(\psi)(\omega_{\nu}) &<& \Exp(S_k(\psi))+\frac{1}{q}\\ \label{KLimits2}
 \frac{1}{n}\sum_{\nu=1}^{n}S_k(\varphi)(\omega_{\nu}) &<& \Exp(S_k(\varphi))+\frac{1}{q}\\ \label{KLimits3}
 \frac{1}{n}\sum_{\nu=1}^{n}\log p_{\omega_{\nu}} &<& \sum_{\tau \in \B}p(\tau)\log p(\tau)+\frac{1}{q}\\
\nonumber &=& -t \left(\Exp(S_k(\psi))+\Exp(S_k(\varphi))\right)+\frac{1}{q}\\
\nonumber &<& -t \left( \frac{1}{n}\sum_{\nu=1}^{n}S_k(\psi)(\omega_{\nu})+\Exp(S_k(\varphi))\right)+ \frac{2}{q}\\
\nonumber &\leq & -t \left(\frac{1}{n}S_{nk}(\psi)(\omega_{\nu})_{\nu=1}^n+\Exp(S_k(\varphi)) \right)+ \frac{2}{q}.
\end{eqnarray}

Clearly we may assume that $(N(q))_{q \in \N}$ is increasing and $N(1) \geq 2$.

Now fix
\begin{eqnarray*}
\zeta &\in& \left(0,\limsup_{r\rightarrow 0}r^{-1}\sum_{A\in C(y,\lambda_r,r)}\diam(A)\right),\\
d &\in& \left(\limsup_{r\rightarrow 0}\lambda_r^{-1}\log r,0\right).
\end{eqnarray*}

We shall now give an inductive construction consisting of a quadruple of rapidly increasing sequences of natural numbers $(\alpha_q)_{q \in \N \cup \{0\}}$, $(\beta_q)_{q \in \N\cup \{0\}}$, $(\gamma_q)_{q \in \N\cup \{0\}}$, $(\lambda_q)_{q \in \N \cup \{0\}}$,  a sequence of positive real numbers $(r_q)_{q \in \N \cup \{0\}}$ and a pair of sequences of finite sets of words $\left(\RR_q\right)_{q \in \N\cup \{0\}}$ and $\left(\Gamma_q\right)_{q \in \N\cup \{0\}}$. First set $\alpha_0=\beta_0=\gamma_0=0$, $\lambda_0=1$ and $\Lambda_0=\Gamma_0 =\emptyset$. For each $q \in \N$ we define 
\begin{eqnarray*}
\alpha_q:= 10k  q^2 \gamma_{q-1} N(q)N(q+1) \bigg\lceil \log \zeta^{-1} c(3+2\rho_{\lambda_{q-1}}) \max \left\lbrace S_{\gamma_{q-1}}(\psi)(\tau)+S_{\gamma_{q-1}}(\varphi)(\tau) : \tau \in \Gamma_{q-1} \right\rbrace  \bigg\rceil.
\end{eqnarray*}
Note that since $\Gamma_{q-1}$ is finite $\alpha_q$ is well defined. 

We then choose $r_q>0$ so that,
\begin{equation}\label{r q def}
-\log r_q > k^{-1}(\alpha_q-\gamma_{q-1}) \left(\Exp(S_k(\varphi))+\frac{1}{q} \right)+\gamma_{q-1}c+q,
\end{equation}
and also 
\begin{equation*}
\sum_{A\in C(y,\lambda_{r_q},r_q)}\diam(A)>\zeta r_q
\end{equation*}
and $\lambda_r^{-1}\log r<d$.

Let $\lambda_q:= \lambda_{r_q}$. We may choose $\RR_q$ to be a finite set of words $\tau \in \A^{\lambda_q}$ so that for each $\tau \in \RR_q$ $\phi_{\tau}([0,1]) \subset B(y,r_q)$ and 
\begin{equation*}
 \sum_{\tau \in \RR_q} \diam \left( \phi_{\tau}([0,1])\right)> \zeta r_q.
 \end{equation*}

Let $\beta_q$ be the largest integer satisfying $k|(\beta_q- \gamma_{q-1})$ and 
\begin{equation}
-\log r_q > k^{-1}(\beta_q-\gamma_{q-1}) \left(\Exp(S_k(\varphi))+\frac{1}{q} \right)+\gamma_{q-1}c+q.
\end{equation}
We let $\gamma_q:= \beta_q+\lambda_q$. We define $\Gamma_q$ by,
\begin{equation*}
\Gamma_q:= \left\lbrace \omega \in \A^{\gamma_q}: \omega|\gamma_{q-1} \in \Gamma_{q-1}, \gamma_{q-1}|\omega|\beta_q \in \B^{k^{-1}(\beta_q- \gamma_{q-1})}, \beta_q|\omega|\gamma_q \in \RR_q \right\rbrace.
\end{equation*}
Note that since $\B$, $\Gamma_{q-1}$ and $\RR_q$ are finite, so is $\Gamma_q$.

We inductively define a sequence of measures $\W_q$ supported on $\Gamma_q$.

For each $\omega \in \A^n$ and $\tau \in \RR_q$ we let
\begin{equation*}
q\left(\omega,\tau\right):=\frac{ \diam \left( \phi_{\omega}\circ \phi_{\tau}([0,1])\right)}{\sum_{\tau \in \RR_q} \diam \left( \phi_{\omega}\circ \phi_{\tau}([0,1])\right)}.
\end{equation*}
Now by the definition of $\Gamma_q$, each $\omega^q \in \Gamma_q$ is of the form $\omega^q=(\omega^{q-1},\kappa_1^q, \cdots, \kappa_{k^{-1}(\beta_q- \gamma_{q-1})}, \tau_q)$ where $\omega^{q-1} \in \Gamma_{q-1}$, $\kappa_{\nu}^q \in \B$ for $\nu=1, \cdots,k^{-1}(\beta_q- \gamma_{q-1})$ and $\tau_q \in \RR_q$. We set,
\begin{equation*}
\W_q(\omega^q)= \W_{q-1}([\omega^{q-1}]) \left(\prod_{\nu=1}^{k^{-1}(\beta_q- \gamma_{q-1})}p(\kappa_{\nu})\right)q\left((\omega^{q-1},\kappa_1^q, \cdots, \kappa_{k^{-1}(\beta_q- \gamma_{q-1})}),\tau_q\right)
\end{equation*}  
Define $\Gamma:= \left\lbrace \omega \in \Sigma: \omega| \gamma_q \in \Gamma_q \text{ for all }q\in \N \right\rbrace$ and extend the sequence $(\W_q)_{q \in \N}$ to a measure $\W$ on $\Gamma$ in the natural way.

We let $S \subseteq \Gamma$ denote the subset,
\begin{equation}
S := \left\lbrace \omega \in \Gamma : [ \gamma_{q-1}|\omega|\beta_q ] \cap S_q \neq \emptyset \text{ for all } q \in \N \right\rbrace.
\end{equation}

\begin{lemma}\label{S interior} For all $\omega \in S$ and $n \in \N$ we have $\pi(\omega) \in \phi_{\omega|n}((0,1))$.
\end{lemma}
\begin{proof}
Suppose for a contradiction that $\omega \in S$ and for some $N \in \N$ $\pi(\omega) \notin \phi_{\omega|N}((0,1))$. Then for all $n \geq N$ we have $\pi(\omega) \in \phi_{\omega|n}(\{0,1\})=\partial \phi_{\omega|n}([0,1])$. However, given $N \in \N$ we may choose $q$ with $\gamma_q>N$. Then $\omega_{\gamma_q+1} \in \B$ by the construction of $S$. Consequently $\phi_{\gamma_q+1}([0,1])$ is in neither the left most, nor the right most interval amongst,
\begin{eqnarray*}
\left\lbrace \phi_{\omega|\kappa(l)}\circ \phi_{\tau}([0,1]): \tau \in \F \right\rbrace.
\end{eqnarray*}
Hence, $\pi(\omega) \notin \partial \phi_{\omega|\gamma_q}([0,1])$. 
\end{proof}

\begin{lemma}\label{subset}
$\pi(S) \subseteq \D_y(\varphi)$.
\end{lemma}
\begin{proof} Take $\omega \in S$. By Lemma \ref{S interior} we have $\pi(\omega) \in \phi_{\omega|n}((0,1)) \subseteq V_{\omega|n}$ and hence $S_n(\varphi)(\omega) \leq S_n(\varphi)(\omega|n)$ for all $n \in \N$ and in particular for each $q \in \N$,
\begin{eqnarray*}
S_{\beta_q}(\varphi)(\omega) & \leq & S_{\beta_q}(\varphi)(\omega|\beta_q)\\
& \leq & S_{\beta_q-\gamma_{q-1}}(\varphi)(\gamma_{q-1}|\omega|\beta_q)+ c \gamma_{q-1}\\
& \leq & \sum_{\nu=1}^{k^{-1}(\beta_q-\gamma_{q-1})}S_{k}(\varphi)(\gamma_{q-1}+(\nu-1)k|\omega|\gamma_{q-1}+\nu k)+ c \gamma_{q-1}.
\end{eqnarray*}
By (\ref{KLimits2}) combined with the fact that $[\gamma_{q-1}|\omega|\beta_q]\cap S_q \neq \emptyset$,
\begin{eqnarray*}
S_{\beta_q}(\varphi)(\omega) & \leq & k^{-1}(\beta_q- \gamma_{q-1})\left(\Exp(S_k(\varphi))+ \frac{1}{q}\right) +c \gamma_{q-1}.
\end{eqnarray*}
Thus, by the definition of $r_q$ we have, $r_q< e^{-S_{\beta_q}(\varphi)(\omega)}$.
\begin{equation*}
T^{\beta_q}(\pi(\omega))=\pi(\sigma^{\beta_q}(\omega)) \in \phi_{\beta_q|\omega|\gamma_q}([0,1])
\end{equation*}
Since $\omega \in S \subseteq \Gamma$, $\beta_q|\omega|\gamma_q \in \RR_q$ and hence
\begin{equation*}
T^{\beta_q}(\pi(\omega))\in \phi_{\beta_q|\omega|\gamma_q}([0,1]) \subseteq B(y,r_q) \subseteq B(y, e^{-S_{\beta_q}(\varphi)(\omega)}).
\end{equation*}
Since this holds for all $q \in \N$, $\pi(\omega) \in \F_y(\varphi)$.
\end{proof}

\begin{lemma} \label{Measure Dim Props} Suppose $\omega \in S$. Given $q \in \N$ and $\gamma_{q-1}< n \leq \beta_q$ we have,
\begin{eqnarray*}
- \log \W_q\left([\omega|n]\right) \geq t\left(S_n(\psi)(\omega|n)+k^{-1}(n-\gamma_{q-1})\Exp(S_k(\varphi))\right)\\
-\frac{3\gamma_{q-1}}{q-1}-2\gamma_{q-1} \rho_{\lambda_{q-1}}-\frac{2n}{q}-N(q)c,
\end{eqnarray*}
\begin{equation*}
- \log \W_q\left([\omega|\gamma_q]\right) \geq tS_{\gamma_q}(\psi)(\omega|\gamma_q)-\frac{3\gamma_{q}}{q}-2\gamma_q \rho_{\lambda_q}.
\end{equation*}
\end{lemma}

\begin{proof}
We prove the lemma by induction. The lemma is trivial for $q=0$. Now suppose that 
\begin{equation*}
- \log \W_{q-1}\left([\omega|\gamma_q]\right) \geq tS_{\gamma_{q-1}}(\psi)(\omega|\gamma_{q-1})-\frac{3\gamma_{q-1}}{q-1}-2\gamma_{q-1} \rho_{\lambda_{q-1}}.
\end{equation*}
Take $\gamma_{q-1}< n \leq \beta_q$ consider $\ell(n):=\lfloor k^{-1}(n-\gamma_{q-1})\rfloor$. If $\ell(n)< N(q)$ then clearly 
\begin{eqnarray*}
S_{n}(\psi)(\omega|n) &\leq& S_{\gamma_{q-1}}(\psi)(\omega|\gamma_{q-1})+S_{n-\gamma_q}(\psi)(\gamma_{q-1}|\omega|n)\\
& \leq & S_{\gamma_q}(\psi)(\omega|\gamma_{q-1})+N(q)c,
\end{eqnarray*}
\begin{eqnarray*}
k^{-1}(n-\gamma_{q-1})\Exp(S_k(\varphi)) \leq N(q) c
\end{eqnarray*}
Since $t<1$ and $N(q-1) \leq N(q)$ it follows from the inductive hypothesis together with the definition of $\W_q$ that,
\begin{eqnarray*}
- \log \W_q\left([\omega|n]\right) & \geq & - \log \W_{q-1}\left([\omega|\gamma_{q-1}]\right) \\
& \geq & t\left(S_n(\psi)(\omega|n)+k^{-1}(n-\gamma_{q-1})\Exp(S_k(\varphi))\right)\\
& & -\frac{3\gamma_{q-1}}{q-1}-2\gamma_{q-1} \rho_{\lambda_{q-1}}-2N(q)c.
\end{eqnarray*}
On the other hand, if $\ell(n) \geq N(q)$ then by equation (\ref{KLimits3}) together with $[ \gamma_{q-1}|\omega|\beta_q ] \cap S_q \neq \emptyset$ we have
\begin{eqnarray*}
\sum_{\nu=k^{-1}\gamma_{q-1}}^{k^{-1}\gamma_q+\ell(n)-1}\log p(\omega_{k\nu+1},\cdots, \omega_{k\nu+k}) &<& - t \left(S_{k\ell(n)}(\psi)(\gamma_{q-1}|\omega|\gamma_{q-1}+k \ell(n))+\ell(n)\Exp(S_k(\varphi)) \right)+ \frac{2n}{q}\\
&<& - t \left(S_{n-\gamma_{q-1}}(\psi)(\omega|n-\gamma_{q-1})+k^{-1}(n-\gamma_{q-1})\Exp(S_k(\varphi)) \right)\\&& +2c+ \frac{2n}{q}.
\end{eqnarray*}
Moreover, by the defintion of $\W_q$ we have,
\begin{eqnarray*}
- \log \W_q\left([\omega|n]\right) & \geq & - \log \W_{q-1}\left([\omega|\gamma_{q-1}]\right) -\sum_{\nu=0}^{\ell(n)-1}\log p(\omega_{k\nu+1},\cdots, \omega_{k\nu+k})\\
&\geq & t\left(S_{\gamma_{q-1}}(\psi)(\omega|\gamma_{q-1})+S_{n-\gamma_{q-1}}(\psi)(\omega|n-\gamma_{q-1})+k^{-1}(n-\gamma_{q-1})\Exp(S_k(\varphi)) \right)\\
&&--\frac{3\gamma_{q-1}}{q-1}-2\gamma_{q-1} \rho_{\lambda_{q-1}}-2c- \frac{2n}{q}\\
&\geq & t\left(S_{n}(\psi)(\omega|n)+k^{-1}(n-\gamma_{q-1})\Exp(S_k(\varphi)) \right)\\
&&-\frac{3\gamma_{q-1}}{q-1}-2\gamma_{q-1} \rho_{\lambda_{q-1}}-N(q)c- \frac{2n}{q}.
\end{eqnarray*}
In particular we have
\begin{eqnarray*}
- \log \W_q\left([\omega|\beta_q]\right) & \geq &  t\left(S_{\beta_q}(\psi)(\omega|\beta_q)+k^{-1}(\beta_q-\gamma_{q-1})\Exp(S_k(\varphi)) \right)\\ &&-\frac{3\gamma_{q-1}}{q-1}-2\gamma_{q-1} \rho_{\lambda_{q-1}}-N(q)c- \frac{2\beta_q}{q}.
\end{eqnarray*}
Note that,
\begin{eqnarray*}
- \log \W_q\left([\omega|\gamma_q]\right) &=& - \log \W_q\left([\omega|\beta_q]\right) - \log q(\omega|\beta_q,\beta_q|\omega|\gamma_q)\\
&=& - \log \W_q\left([\omega|\beta_q]\right) - \log \left(\frac{\diam \left(\phi_{\omega|\gamma_q}([0,1])\right) }{\sum_{\tau \in \RR_q} \diam \left(\phi_{\omega|\beta_q} \circ \phi_{\tau}([0,1])\right)}\right)\\
&\geq & - \log \W_q\left([\omega|\beta_q]\right) - t\log \left(\frac{\diam \left(\phi_{\omega|\gamma_q}([0,1])\right) }{\sum_{\tau \in \RR_q} \diam \left( \phi_{\omega|\beta_q} \circ \phi_{\tau}([0,1])\right)}\right).
\end{eqnarray*}
Clearly,
\begin{eqnarray*}
- \log \diam \left(\phi_{\omega|\gamma_q}([0,1])\right) \geq S_{\gamma_q}(\psi)(\omega|\gamma_q)-\gamma_q \rho_{\gamma_q}
\end{eqnarray*}
Moreover,
\begin{eqnarray*}
\sum_{\tau \in \RR_q} \diam \left( \phi_{\omega|\beta_q \circ \tau}([0,1])\right)& \geq &
\sum_{\tau \in \RR_q} \exp\left(-S_{\gamma_q}(\psi)(\omega|\beta_q,\tau)\right)\\
& \geq & e^{-S_{\beta_q}(\psi)(\omega|\beta_q)} \sum_{\tau \in \RR_q} e^{-S_{\lambda_q}(\psi)(\tau)}\\
&\geq & e^{-S_{\beta_q}(\psi)(\omega|\beta_q)-\lambda_q \rho_{\lambda_q}} \sum_{\tau \in \RR_q} \diam \left( \phi_{\tau}([0,1]) \right)\\
& \geq & e^{-S_{\beta_q}(\psi)(\omega|\beta_q)-\lambda_q \rho_{\lambda_q}} \zeta r_q.
\end{eqnarray*}
Note that from the definition of $\beta_q$ and $c$ we have,
\begin{eqnarray*}
- \log r_q \leq  k^{-1}(\beta_q-\gamma_{q-1})\Exp(S_k(\varphi))+c(\gamma_{q-1}+1)+q
\end{eqnarray*}
Combining these inequalities we see that,
\begin{eqnarray*}
- \log \W_q\left([\omega|\gamma_q]\right) &\geq &tS_{\gamma_q}(\psi)(\omega|\gamma_q)-\gamma_q \rho_{\gamma_q}- N(q)c -\frac{2\beta_q}{q}\\&& -\frac{3\gamma_{q-1}}{q-1}-2\gamma_{q-1} \rho_{\lambda_{q-1}}-\lambda_q \rho_{\lambda_q}-c(\gamma_{q-1}+1)-q + \log \zeta\\
&\geq &tS_{\gamma_q}(\psi)(\omega|\gamma_q)-\frac{3\gamma_{q}}{q}-2\gamma_{q} \rho_{\lambda_{q}},
\end{eqnarray*}
since $\gamma_q\geq \beta_q \geq \alpha_q$ and by the definition of $\alpha_q$, 
\begin{equation*}
\alpha_q>q \left( \frac{3\gamma_{q-1}}{q-1}+2\gamma_{q-1} \rho_{\lambda_{q-1}}+c(\gamma_{q-1}+1)+q -\log \zeta \right).
\end{equation*}
\end{proof}

We define a Borel measure $\mu$ by $\mu(A):=\W(S \cap \pi^{-1}(A))$ for Borel sets $A \subseteq [0,1]$.

\begin{lemma}\label{mu has positive measure} $\mu([0,1])>0$. 
\end{lemma}

\begin{proof}
This follows immediately from the fact that 
\begin{eqnarray*}
\W(S) \geq \prod_{q \in \N}(1-\delta_q)>0.
\end{eqnarray*}
\end{proof}

\begin{lemma}\label{Pointwise Dimension Lemma}For all $\omega \in S$ we have 
\begin{equation*}
\liminf_{r \rightarrow 0} \frac{ \log \mu( B(\pi(\omega),r) )}{\log r} \geq t.
\end{equation*}
\end{lemma}
\begin{proof}
For the proof of Lemma \ref{Pointwise Dimension Lemma} we shall require some additional notation. Given a pair of functions $f$ and $g$, depending on $q \in \N$ and $r \in (0,1)$, we shall write,
\begin{equation}
f(q,r)\geq g(q,r)-\Error,
\end{equation}
to denote that for each $\epsilon>0$ there exists an $N \in \N$ and a $\delta>0$ such that given any $(q, r) \in \N\times (0,1)$ with $q>N$ and $r< \delta$ we have 
\begin{equation}
f(q,r) \geq g(q,r)-\epsilon.
\end{equation}
Note that by (\ref{r q def}) $r_q<e^{-q}$ for all $q \in \N$ and by Definition \ref{PUCD} this implies that $\lim_{q\rightarrow \infty} \lambda_{q}=\lim_{q\rightarrow \infty} \lambda_{r_q}=\infty$ and hence $\lim_{q \rightarrow \infty} \rho_{\lambda_q}=0$. Thus for any function $g:\N \times (0,1) \rightarrow \R$, 
\begin{equation*}
g(q,r)-\rho_{\lambda_q}\geq g(q,r)- \Error.
\end{equation*}
Similarly, it follows from the definition of $\beta_q$ that  
\begin{equation*}
g(q,r)-cN(q)N(q+1)\beta_q^{-1}\geq g(q,r)- \Error.
\end{equation*}

Firstly we show that for any $x= \pi(\omega)$ with $\omega \in S$ $B(x,r)$ and $r>0$ for which there exists $q \in \N$ and $l \in \N$ with $\gamma_{q-1} \leq l < \beta_q$ such that
\begin{eqnarray*}
B(x,r)\cap \pi(S) \subseteq \phi_{\omega|l}([0,1]) \text{ but }  B(x,r)\cap \pi(S) &\not\subseteq& \phi_{\omega|l+1}([0,1])
\end{eqnarray*} 
satisfies
\begin{eqnarray}\label{LD1}
\frac{\log \mu ( B(x,r))}{\log r} &\geq & t- \Error.
\end{eqnarray}
Indeed, as $B(x,r)\cap \pi(S) \subseteq \phi_{\omega|l}([0,1])$ it follows from Lemma \ref{Measure Dim Props} that,
\begin{eqnarray*}
-\log \mu ( B(x,r)) & \geq &-  \log \W \left([\omega|l]\right) \\
&= & - \log \W_q\left([\omega|l]\right)\\& \geq &t S_l(\psi)(\omega|l)
-\frac{3\gamma_{q-1}}{q-1}-2\gamma_{q-1} \rho_{\lambda_{q-1}}-\frac{2l}{q}-N(q)c\\
&= & - \log \W_q\left([\omega|l]\right)\\& \geq &t S_l(\psi)(\omega|l)
-\frac{6l}{q-1}-2l \rho_{\lambda_{q-1}},
\end{eqnarray*}
since $l \geq \gamma_{q-1}>qN(q)c$.
Since $S_l(\psi)(\omega|l)\geq l\log \xi$ this implies
\begin{eqnarray*}
\frac{\log \mu ( B(x,r))}{S_l(\psi)(\omega|l)} &\geq &t -\log \xi^{-1}\left(\frac{6}{q-1}+2 \rho_{\lambda_{q-1}}\right).
\end{eqnarray*}

However, $B(x,r)\cap \pi(S)\not\subseteq \phi_{\omega|l+1}([0,1])$ and hence $B(x,r)\cap \pi(S) \not\subseteq \phi_{\omega|\kappa(l)}([0,1])$ where $\kappa(l):=k \lceil k^{-1}(l+1)\rceil$. It follows that $B(x,r)\cap \pi(S)$ intersects $\phi_{\tau|\kappa(l)}([0,1])$, for some $\tau \in S$, as well as $\phi_{\omega|\kappa(l)}([0,1])$. Since $\kappa(l) \leq \beta_q$ and $\omega,\tau \in S$, $(\kappa(l)-k)|\omega|\kappa(l),(\kappa(l)-k)|\tau|\kappa(l) \in \B$. Thus, by (\ref{D Distance Equation}),
\begin{eqnarray*} r &\geq& \frac{1}{2} e^{-S_n(\psi)(\omega|\kappa(l)-k)-c}\\
&\geq& e^{-S_n(\psi)(\omega|l)-c-\log 2}.
\end{eqnarray*}
Thus,
\begin{eqnarray*}
\frac{\log \mu ( B(x,r))}{\log r} &\geq &\left(1+\frac{c+\log 2}{\log r}\right)\left(t -\log \xi^{-1}\left(\frac{6}{q-1}+2 \rho_{\lambda_{q-1}}\right)\right)
\end{eqnarray*}
which implies the first claim (\ref{LD1}).

Secondly, we show that given $\omega \in S$, $x \in [0,1]$ and $r>0$ for which $B(x,r)\cap \pi(S)\subseteq \phi_{\omega|\beta_q}([0,1])$ and yet $B(x,r)\cap \pi(S) \not\subseteq \phi_{\omega|\beta_q}\circ \phi_{\tau}([0,1])$ for any $\tau \in \RR_q$ we have,
\begin{eqnarray}\label{LD2}
\frac{ \log \mu\left(B(x,r)\right)}{\log r} \geq t- \Error.
\end{eqnarray}

From the proof of Lemma \ref{Measure Dim Props} we have,
\begin{eqnarray*}
- \log \W_q\left([\omega|\beta_q]\right) & \geq &  t\left(S_{\beta_q}(\psi)(\omega|\beta_q)+k^{-1}(\beta_q-\gamma_{q-1})\Exp(S_k(\varphi)) \right)\\ &&-\frac{3\gamma_{q-1}}{q-1}-2\gamma_{q-1} \rho_{\lambda_{q-1}}-N(q)c- \frac{2\beta_q}{q}\\
- \log r_q &\leq&  k^{-1}(\beta_q-\gamma_{q-1})\Exp(S_k(\varphi))+c(\gamma_{q-1}+1)+q\\
\sum_{\tau \in \RR_q} \diam \left( \phi_{\omega|\beta_q \circ \tau}([0,1])\right)& \geq &
 e^{-S_{\beta_q}(\psi)(\omega|\beta_q)-\lambda_q \rho_{\lambda_q}} \zeta r_q.
\end{eqnarray*}
Suppose $r>r_q$. Then by the first two inequalities together with the fact that $B(x,r) \subseteq \phi_{\omega|\beta_q}([0,1])$ we have
\begin{eqnarray*}
-\log \mu(B(x,r)) & \geq & - \log \W_q\left([\omega|\beta_q]\right) \\
& \geq & - t \log r-\left(\frac{3\gamma_{q-1}}{q-1}+2\gamma_{q-1} \rho_{\lambda_{q-1}}+N(q)c+\frac{2\beta_q}{q}+c(\gamma_{q-1}+1)+q\right).
\end{eqnarray*}
Note also that $B(x,r) \subseteq \phi_{\omega|\beta_q}([0,1])$ implies $- \log r>\beta_q \log \xi > \gamma_{q-1} \log \xi$ and hence,
\begin{eqnarray*}
\frac{\log \mu(B(x,r))}{\log r} & \geq & t - \log \xi^{-1}\left(\frac{3}{q-1}+2\rho_{\lambda_{q-1}}+\frac{N(q)c+c(\gamma_{q-1}+1) +q}{\beta_q}+\frac{2}{q}\right)\\
& \geq & t- \Error.
\end{eqnarray*}

Now suppose that $r\leq r_q$ and let $\T$ denote the following collection,
\begin{eqnarray*}
\T &:=&\left\lbrace \tau \in \RR_q: \frac{\diam\left(\phi_{\omega|\beta_q} \circ \phi_{\tau}([0,1]) \cap B(x,r)\right)}{\diam \left(\phi_{\omega|\beta_q} \circ \phi_{\tau}([0,1])\right)}> \frac{1}{2}  \right\rbrace.
\end{eqnarray*}
We also define $B_{\T}(x,r) \subseteq B(x,r)$ by,
\begin{eqnarray*}
B_{\T}(x,r):= \bigcup_{\tau \in \T} \phi_{\omega|\beta_q} \circ \phi_{\tau}([0,1]) 
\end{eqnarray*}

From the definition of $\mu$ and $\W$ we see that for each $\tau \in \RR_q$ we have,
\begin{eqnarray*}
\mu(\phi_{\omega|\beta_q}\circ \phi_{\tau}([0,1])) & \leq & \W_q\left([\omega|\beta_q, \tau]\right)\\
 &\leq & \W_q\left([\omega|\beta_q]\right) \cdot \frac{\diam \left(\phi_{\omega|\beta_q}\circ \phi_{\tau}([0,1])\right) }{\sum_{\tau \in \RR_q} \diam \left( \phi_{\omega|\beta_q} \circ \phi_{\tau}([0,1])\right)}.
\end{eqnarray*}
Hence, as $t<1$,
\begin{eqnarray*}
\mu(B_{\T}(x,r))& \leq & \sum_{\tau \in \T}\mu(\phi_{\omega|\beta_q}\circ \phi_{\tau}([0,1]))\\
& \leq & \W_q\left([\omega|\beta_q]\right) \cdot \frac{ \sum_{\tau \in \T}\diam \left(\phi_{\omega|\beta_q}\circ \phi_{\tau}([0,1])\right) }{\sum_{\tau \in \RR_q} \diam \left( \phi_{\omega|\beta_q} \circ \phi_{\tau}([0,1])\right)}
\\
& \leq & \W_q\left([\omega|\beta_q]\right) \left( \frac{ \sum_{\tau \in \T}\diam \left(\phi_{\omega|\beta_q}\circ \phi_{\tau}([0,1])\right) }{\sum_{\tau \in \RR_q} \diam \left( \phi_{\omega|\beta_q} \circ \phi_{\tau}([0,1])\right)}\right)^t
\\
& \leq & 2\W_q\left([\omega|\beta_q]\right) \left( \sum_{\tau \in \RR_q} \diam \left( \phi_{\omega|\beta_q} \circ \phi_{\tau}([0,1])\right)\right)^{-t}r^t.
\end{eqnarray*}
Piecing the previous inequalities together with the observations from the proof of Lemma \ref{Measure Dim Props} we obtain
\begin{eqnarray*}
- \log \mu\left(B_{\T}(x,r) \right)
\end{eqnarray*}
\begin{eqnarray*}
 \geq -t \log r -\left(\frac{3\gamma_{q-1}}{q-1}+2\gamma_{q-1} \rho_{\lambda_{q-1}}+N(q)c+ \frac{2\beta_q}{q}+c(\gamma_{q-1}+1)+q+ \lambda_q \rho_{\lambda_q} - \log \zeta - \log 2\right).
\end{eqnarray*}
Now $\lambda_q<d \log r_q \leq d \log r$, where $d<0$ is the constant as appears in the positive upper cylinder density condition. Hence,
\begin{eqnarray}\label{LD4} 
\frac{ \log \mu\left(B_{\T}(x,r)\right)}{\log r}
\end{eqnarray}
\begin{eqnarray*} &\geq & \nonumber  t -\log \xi ^{-1}\left(\frac{3}{q-1}+2 \rho_{\lambda_{q-1}}+\frac{N(q)c+c(\gamma_{q-1}+1)+q- \log \zeta+ \log 2}{\beta_q}+ \frac{2}{q} \right)+d \rho_{\lambda_q}\\
& \geq & t- \Error.
\end{eqnarray*}

Consider the set $\C:=\left\lbrace \tau \in \RR_q: \phi_{\omega|\beta_q} \circ \phi_{\tau}([0,1]) \cap B(x,r) \neq \emptyset, \tau \notin \T \right\rbrace$. It is clear that $\C$ contains at most two elements, with $\phi_{\omega|\beta_q} \circ \phi_{\tau}([0,1])$ containing either $\inf B(x,r)$ or $\sup B(x,r)$. We shall show that for $\tau \in \C$ we have,
\begin{eqnarray}\label{LD3}
\frac{\log \mu\left((\phi_{\omega|\beta_q} \circ \phi_{\tau})([0,1]) \cap B(x,r)\right)}{\log r}
& \geq & t- \Error.
\end{eqnarray}
Take $\tau \in \C$ and assume that $\sup B(x,r) \in \phi_{\omega|\beta_q} \circ \phi_{\tau}([0,1])$ ie. $ \phi_{\omega|\beta_q} \circ \phi_{\tau}([0,1])$ intersects the right hand boundary of $B(x,r)$. Since $\tau \notin \T$ we have $\diam\left(\phi_{\omega|\beta_q} \circ \phi_{\tau}([0,1]) \cap B(x,r)\right)< \frac{1}{2}\diam\left(\phi_{\omega|\beta_q} \circ \phi_{\tau}([0,1])\right).$ Choose $\tilde{\omega}\in S$ such that $\pi(\tilde{\omega})$ is on the right hand side of $\phi_{\omega|\beta_q} \circ \phi_{\tau}([0,1]) \cap B(x,r)\cap \pi(S)$. Define $\tilde{r}:=|\pi(\tilde{\omega})-\inf (\phi_{\omega|\beta_q} \circ \phi_{\tau})([0,1])|,$
and consider $B(\pi(\tilde{\omega}), \tilde{r})$. Since $\pi(\tilde{\omega})$ is on the right hand side of $(\phi_{\omega|\beta_q} \circ \phi_{\tau})([0,1]) \cap B(x,r)\cap \pi(S)$ and 
\begin{eqnarray*}
\diam\left(\phi_{\omega|\beta_q} \circ \phi_{\tau}([0,1]) \cap B(x,r)\right)<\frac{1}{2}\diam\left(\phi_{\omega|\beta_q} \circ \phi_{\tau}([0,1])\right),
\end{eqnarray*}
we have 
\begin{eqnarray*}
(\phi_{\omega|\beta_q} \circ \phi_{\tau})([0,1]) \cap B(x,r) \cap \pi(S) \subseteq B(\pi(\tilde{\omega}), \tilde{r}) \subseteq (\phi_{\omega|\beta_q} \circ \phi_{\tau})([0,1])
\end{eqnarray*}
and $\tilde{\omega}|\gamma_q=(\omega|\beta_q, \tau)$. 

We consider two cases. First suppose that $B(\pi(\tilde{\omega}),\tilde{r})\subseteq \phi_{\tilde{\omega}|\beta_{q+1}}([0,1])$. It follows from Lemma \ref{Measure Dim Props} that,
\begin{eqnarray*}
-\log \mu\left(B(\pi(\tilde{\omega}),\tilde{r})\right)&\geq & -\log \W_{q+1}\left([\tilde{\omega}|\beta_{q+1}]\right)\\
& \geq & t\left(S_{\beta_{q+1}}(\psi)(\omega|\beta_{q+1})+k^{-1}(\beta_{q+1}-\gamma_{q-1})\exp(S_k(\varphi))\right)\\
&&-\frac{3\gamma_{q}}{q}-2\gamma_{q} \rho_{\lambda_{q}}-\frac{2\beta_{q+1}}{q+1}-N(q+1)c\\
& \geq & t \beta_{q+1} \log \xi -\left(k \log \xi+cN(q+1)+\frac{5\beta_{q+1}}{q}+2\beta_{q+1}\rho_{\lambda_q}\right).
\end{eqnarray*}
Hence,
\begin{eqnarray*}
\frac{-\log \mu\left((\phi_{\omega|\beta_q} \circ \phi_{\tau})([0,1]) \cap B(x,r)\right)}{\beta_{q+1}\log \xi}
& \geq & t  -\log \xi^{-1}\left(\frac{k \log \xi+cN(q+1)}{\beta_{q+1}}+\frac{5}{q}+2\rho_{\lambda_q}\right).
\end{eqnarray*}
Since $B(x,r) \cap \pi(S) \not\subseteq (\phi_{\omega|\beta_q}\circ \phi_{\tau'})([0,1])$ for any $\tau' \in \RR_q$, it follows from (\ref{D Distance Equation}) that
\begin{eqnarray}\label{r and Q}
-\log r & \leq & -\max \left\lbrace S_{\gamma_{q}}(\psi)(\tau'): \tau' \in \Gamma_{q} \right\rbrace-c\\ \nonumber
& \leq & \alpha_{q+1} \log \xi<\beta_{q+1} \log \xi.
\end{eqnarray}
Thus,
\begin{eqnarray*}
\frac{\log \mu\left((\phi_{\omega|\beta_q} \circ \phi_{\tau})([0,1]) \cap B(x,r)\right)}{\log r}
& \geq & t  -\log \xi^{-1}\left(\frac{k \log \xi+cN(q+1)}{\beta_{q+1}}+\frac{5}{q}+2\rho_{\lambda_q}\right)\\
& \geq & t  -\Error.
\end{eqnarray*}
Now suppose that $B(\pi(\tilde{\omega}),\tilde{r})\not\subseteq \phi_{\tilde{\omega}|\beta_{q+1}}([0,1])$. Then we may apply (\ref{LD1}) to obtain
\begin{eqnarray}
\frac{\log \mu( B(\pi(\tilde{\omega},\tilde{r}))}{\log \tilde{r}} & \geq & t- \eta(q+1, \tilde{r}).
\end{eqnarray}
Clearly $\tilde{r}<2r$ and so $\lim_{r \rightarrow \infty}\frac{\log \tilde{r}}{\log r}\geq 1$ and hence,
\begin{eqnarray*}
\frac{\log \mu\left((\phi_{\omega|\beta_q} \circ \phi_{\tau})([0,1]) \cap B(x,r)\right)}{\log r}& \geq & t- \Error.
\end{eqnarray*}
By symmetry the same holds if $ \phi_{\omega|\beta_q} \circ \phi_{\tau}([0,1])$ intersects the left hand boundary of $B(x,r)$. This proves the claim (\ref{LD3}).

Recall that,
\begin{eqnarray*}
B(x,r)\cap\pi(S) \subseteq  B_{\T}(x,r) \cup \left(\bigcup_{\tau \in \C}(\phi_{\omega|\beta_q} \circ \phi_{\tau})([0,1]) \cap B(x,r) \right).
\end{eqnarray*}
Noting that $\#\C\leq 2$ we obtain,
\begin{eqnarray*}
\mu\left(B(x,r)\right) & \leq & \mu\left(B_{\T}(x,r)\right)+ \sum_{\tau \in \C}\mu\left((\phi_{\omega|\beta_q} \circ \phi_{\tau})([0,1]) \cap B(x,r)\right)\\
& \leq & 3 \max\left\lbrace \mu\left(B_{\T}(x,r)\right) \right\rbrace \cup \left\lbrace \mu\left((\phi_{\omega|\beta_q} \circ \phi_{\tau})([0,1]) \cap B(x,r)\right): \tau \in \C \right\rbrace.
\end{eqnarray*}

By combining with (\ref{LD4}) and (\ref{LD3}), 
\begin{eqnarray*}
\frac{ \log \mu\left(B(x,r)\right)-\log 3}{\log r} & \geq & t-\Error,
\end{eqnarray*}
which implies (\ref{LD2}).

To complete the proof of the Lemma we fix $\omega \in S$, let $x= \pi(\omega)$ and consider a ball $B(\pi(\omega),r)$ of radius $r>0$. Now choose $q(r) \in \N$ so that 
\begin{eqnarray*}
B(x,r)\cap \pi(S) \subseteq \phi_{\omega|\gamma_{q(r)-1}}([0,1]) \text{ but }  B(x,r)\cap \pi(S) &\not\subseteq& \phi_{\omega|\gamma_{q(r)}}([0,1]).
\end{eqnarray*} 
Now either $B(x,r)\cap \pi(S) \not\subseteq \phi_{\omega|\beta_{q(r)}}([0,1])$, in which case we apply (\ref{LD1}) or $B(x,r)\cap \pi(S) \not\subseteq \phi_{\omega|\beta_{q(r)}}([0,1])$ in which case we apply (\ref{LD2}). In both cases we obtain,
\begin{eqnarray}\label{almost done}
\frac{\log \mu(B(x,r))}{\log r} \geq t - \eta(q(r),r).
\end{eqnarray}
By (\ref{r and Q}) whenver $q(r) \leq Q$ we have 
\begin{equation*}
r \geq \exp\left( -\max \left\lbrace S_{\gamma_{Q}}(\psi)(\tau'): \tau' \in \Gamma_{Q} \right\rbrace-c\right)>0.
\end{equation*}
Hence, $\lim_{r \rightarrow 0} q(r)= \infty$. Therefore, by (\ref{almost done}) we have
\begin{eqnarray}
\liminf_{r \rightarrow 0}\frac{\log \mu(B(\pi(\omega),r))}{\log r} \geq t.
\end{eqnarray}
\end{proof}
To complete the proof of Proposition \ref{Technical Lemma} we recall the following standard Lemma.

\begin{lemma}\label{dim lem} Let $\nu$ be a finite Borel measure on some metric space $X$. Suppose we have $J\subseteq X$ with $\nu(J)>0$ such that for all $x \in J$
 \[\liminf_{r \rightarrow 0}\frac{\log \nu( B(x,r))}{\log r} \geq d.\]Then $\dim J \geq d$.
\end{lemma}
\begin{proof}
See \cite[Proposition 2.2]{F2} .
\end{proof}
Thus by Lemmas \ref{Pointwise Dimension Lemma} and \ref{mu has positive measure} we have
\begin{equation*}
\dim \pi(S) \geq t >s.
\end{equation*}
Hence, by Lemma \ref{subset} the Hausdorff dimension of $\D_y(\varphi)$ is at least $s$. Since this for all $s<s(\varphi)$, we have
\begin{equation*}
\dim \D_y(\varphi)\geq s(\varphi).
\end{equation*}
\end{proof}

\section{Proof of Theorem \ref{CE}}\label{CES}

\begin{proof}[Proof of Theorem \ref{CE}.]
We begin by defining a sequence $(r_n)_{n \in \N}$ by
\begin{equation}
r_n:=\min\left\lbrace \left(2+\sum_{q \in \N}e^{-q/n}\right)^{-n^2}\cdot e^{-2n^2}, \frac{1}{2}\left(\Phi(n)-\Phi(n+1)\right)\right\rbrace.
\end{equation}
Note that since $\Phi$ is strictly decreasing each $r_n>0$. Now take $n_0 > 2$ so that $\Phi(n_0)< \left(1- 2^{1-\beta^{-1}}\right)$ and $\sum_{n \geq n_0}e^{-\beta n}<1$. For each $n \geq n_0$ we choose some closed interval $V_n \subset (\Phi_{n+1},\Phi_n)$ of length $r_n$, which is always possible, since $r_n<\Phi(n)-\Phi(n+1)$. Note that since each $r_n<e^{-n}$ we have $\sum_{n\geq n_0}r_n^{\beta}\leq \sum_{n \geq n_0}e^{- \beta n}<1$. Hence, $r_1=r_2:=2^{-\beta^{-1}}\left(1-\sum_{n \geq n_0}r_n^{\beta}\right)^{\beta^{-1}}>0$. Note also that $1-\Phi(n_0)>2^{1-\beta^{-1}}>2r_1$. Thus, we may choose two disjoint closed intervals $V_1, V_2$ of width $r_1=r_2$ contained within $(\Phi(n_0),1)$.

We now let $\A:=\left\lbrace n \in \N: n \geq n_0 \right\rbrace \cup \left\lbrace 1, 2\right\rbrace$. Define $T:\bigcup_{n \in \A} V_n \rightarrow [0,1]$ to be the unique expanding Markov map which maps each of the intervals $\{V_n\}_{n\in \A}$ onto $[0,1]$ in an affine and orientation preserving way. First note that,
\begin{equation}
\sum_{n\in \A} \diam(V_n)^{\beta}=r_1^{\beta}+r_2^{\beta}+\sum_{n \geq n_0}r_n^{\beta}=1.
\end{equation}
Thus, $\dim \Lambda= \beta$ by Moran's formula.

Take $n \geq n_0$ and consider $\Sh^{(n)}_0(\Phi):=\left\lbrace x \in \Lambda: |T^n(x)|< \Phi(n) \right\rbrace$. Since $T$ is orientation preserving it follows from the construction of $T$ that we can cover $S_n(\Phi)$ with sets of the form $V_{\omega}=\cap_{j=0}^{n}T^{-j}V_{\omega_j}$ where $\omega \in \mathcal{C}_n:=\left\lbrace \omega \in \A^{n+1}: \omega_{n+1}\geq n \right\rbrace$. Since $T$ is piecewise linear we have $\diam V_{\omega}= \prod_{j=1}^{n+1} r_{\omega_j}$ for each $\omega \in \A^{n+1}$. It follows that for any $m>n_0$ we may cover $\Sh_0(\Phi)$ with the family $\bigcup_{n \geq m}\left\lbrace V_{\omega}: \omega \in \mathcal{C}_n\right\rbrace$.

Now take $\epsilon>0$. For all $n>\epsilon^{-1}$ we have,
\begin{eqnarray*}
\sum_{\omega \in \mathcal{C}_n} \left( \diam V_{\omega}\right)^{\epsilon} & \leq & \sum_{\omega \in \mathcal{C}_n} (r_{\omega_1} \cdots r_{\omega_n})^{\epsilon}\\
& = & \left(\sum_{n \in \A} r_n^{\epsilon}\right)^n \cdot \sum_{q \geq n}r_n^{\epsilon}\\
& \leq & \left(2+\sum_{q \in \N} e^{-\epsilon q}\right)^n \cdot \sum_{k \geq n} \left(\left(2+\sum_{q \in \N}e^{-q/k}\right)^{-k^2}\cdot e^{-2k^2}\right)^{\epsilon}\\
& \leq & \left(2+\sum_{q \in \N} e^{-\epsilon q}\right)^n \cdot \left(2+\sum_{q \in \N}e^{-q/n}\right)^{-n^2 \epsilon } \cdot  \sum_{k \geq n} e^{-2kn \epsilon}\\
& \leq & \left(2+\sum_{q \in \N} e^{-\epsilon q}\right)^n \cdot \left(2+\sum_{q \in \N}e^{-q/n}\right)^{-n } \cdot e^{-n} \sum_{k \geq n} e^{-k}\\
& \leq & e^{-n} \sum_{k \in \N} e^{-k}.
\end{eqnarray*}
Thus, for all $m>\epsilon^{-1}$ we have,
\begin{eqnarray*}
\sum_{n \geq m} \sum_{\omega \in \mathcal{C}_n} \left(\diam V_{\omega}\right)^{\epsilon}  \leq \sum_{n \geq m} e^{-n} \sum_{k \in \N} e^{-k} \leq \left(\sum_{k \in \N} e^{-k}\right)^2< \infty.
\end{eqnarray*}
Since $\lim_{m \rightarrow \infty} \sup  \left\lbrace \diam V_{\omega}: \omega \in \mathcal{C}_n \right\rbrace =0$ it follows that $\dim S_0(\Phi)< \epsilon$. As this holds for all $\epsilon>0$ we have $\dim S_0(\Phi)=0$.
\end{proof}

We note that by Corollary \ref{positive corollary} $s(\alpha)>0$ for all $\alpha \in \R_{> 0}$.

\section{Remarks}\label{remarks}

Both Theorems \ref{ifs} and \ref{dense unit interval} may be extended in a number of ways with some minor alterations of the proof.

Given $\Phi:\N \times \Lambda \rightarrow (0,1)$ we define 
\begin{equation*}
\Ss_y(\Phi):=\bigcap_{m \in \N} \bigcup_{n \geq m}\left\lbrace x \in \Lambda: |T^n(x)-y|<\Phi(n,x)\right\rbrace.
\end{equation*}
Theorems \ref{ifs} and \ref{dense unit interval} both deal with the case where $\Phi$ is multiplicative, ie. $\Phi(n+m,x)=\Phi(n,T^m(x))\cdot \Phi(m,x)$, for all $n,m \in \N\cup\{0\}$ and $x \in \Lambda$. Indeed, when $\Phi$ is multiplicative, we may take $\varphi: x\mapsto -\log \Phi(0,x)$ so that $\Phi(n,x)=\exp(-S_n(\varphi)(x))$ and $\Ss_y(\Phi)=\D_y(\varphi)$. 

We say that $\Phi$ is almost multiplicative if there exists some constant $C>1$ such that,
\begin{equation*}
C^{-1}<\frac{\Phi(n,T^m(x))\cdot \Phi(m,x)}{\Phi(n+m,x)}<C,
\end{equation*}
for all $n,m \in \N$ and $x \in \Lambda$. Examples include the norms of certain matrix products (see \cite{FL,IY}). Given $\omega \in \A^n$ we let $\Phi(\omega):= \sup\left\lbrace \Phi(n,x):x \in V_{\omega}\right\rbrace$. Following Feng and Lau \cite{FL} one may define a pressure function, $P(s,\Phi)\rightarrow \R$ by
\begin{equation*}
P(s,\Phi):=\lim_{n \rightarrow \infty}\frac{1}{n} \sum_{\omega \in \A^n}\left(\Phi(\omega)\cdot||\psi_{\omega}'||_{\infty}\right)^s,
\end{equation*}
and let $s(\Phi):= \inf \left\lbrace s: P(s,\Phi)\leq 0\right\rbrace$. Technical modifications to the proof of Theorems \ref{ifs} and \ref{dense unit interval} show that whenever $T$ is a countable Markov map and $\Phi$ is almost multiplicative, $\dim \Ss_y(\Phi)=s(\Phi)$ for all $y \in \Lambda$, and if $\overline{\Lambda}=[0,1]$ then $\dim \Ss_y(\Phi)=s(\Phi)$ for all $y \in \overline{\Lambda}$.

Instead of considering the sets $\D_y(\varphi)$ we can consider sets of the form,
\begin{eqnarray*}
\Lo_y(\varphi):= \left\lbrace x \in \Lambda : \limsup_{n \rightarrow \infty} \frac{\log d(T^n(x),y)}{S_n(\varphi)(x)}=-1 \right\rbrace.
\end{eqnarray*}
When $T$ is a countable Markov map we have $\dim \Lo_y(\varphi)=\dim \D_y(\varphi)=s(\varphi)$ for all $y \in \Lambda$ and when $T$ is a countable Markov map satisfying $\overline{\Lambda}=[0,1]$ we have $\dim \Lo_y(\varphi)=\dim \D_y(\varphi)=s(\varphi)$ for all $y \in [0,1]$. To prove the upper bound we note that $\Lo_y(\varphi)\subset \dim \D_y((1-\delta)\varphi)$ for all $\delta\in (0,1)$ and $\lim_{\delta \rightarrow 0}\dim \D_y((1-\delta)\varphi) =\lim_{\delta \rightarrow 0}s((1-\delta)\varphi)=s(\varphi)$. To prove the lower bound requires a technical adaptation of the proof of Proposition \ref{Technical Lemma}, removing those points $x$ for which $T^n(x)$ moves too close to $y$.

One can also consider what happens when we replace assumption (1) in Definition \ref{EMR def} with the weaker assumption that $T$ is modelled by a subshift of finite type. If the corresponding matrix is finitely primitive (see \cite[Section 2.1]{GDMS}) then one may adapt the proofs of Theorems \ref{ifs} and \ref{dense unit interval} with only mino modifications. However, to determine the dimension of $\D_y(\varphi)$ for an arbitrary countable subshift of finite type would require further innovation.

\end{document}